\documentclass[10pt,letterpaper]{article}
\usepackage{amsfonts}
\usepackage{amssymb}
\usepackage{amsmath}
\usepackage{amsthm}
\usepackage{latexsym}
\newlength{\defbaselineskip}
\newcommand{\setlinespacing}[1]{\setlength{\baselineskip}{#1 \defbaselineskip}}

5.8in \numberwithin{equation}{section}

\newtheorem{thm}{\indent Theorem}

\newtheorem{cor}{\indent Corollary}
\newtheorem{lem}{\indent Lemma}
\newtheorem{prop}{\indent Proposition}
\theoremstyle{remark}
\newtheorem{rem}{\indent Remark}
\theoremstyle{definition}
\newtheorem{defn}{\indent Definition}

\newcommand{\X}{\mathcal{X}}

\newcommand{\R}{\mathbb{R}}

\newcommand{\N}{\mathbb{N}}

\newcommand{\ynn}{Y^{(n)}}

\newcommand{\set}[1]{\left\{#1\right\}}

\newcommand{\du}{{d}u}

\newcommand{\EE}{\mathbb{E}}

\newcommand{\PP}{\mathbb{P}}
\newcommand{\QQ}{\mathbb{Q}}

\newenvironment{sys*}{\begin{eqnarray*}\left\{\begin{array}{l}}{\end{array}\righ
t.\end{eqnarray*}}

\newcommand{\MM}{{\mathcal M}}

\newcommand{\CC}{{\mathcal C}}

\newcommand{\RR}{{\mathcal R}}

\newcommand{\FE}{{\mathbb F}}
\newcommand{\PO}{{\mathcal P}}
\newcommand{\SU}{{\mathcal S}}
\newcommand{\SUb}{\SU_1}
\newcommand{\DE}{{\mathcal V}}

\newcommand{\FF}{{\mathcal F}}
\newcommand{\FFt}{\FF_t}
\newcommand{\GG}{{\mathcal G}}
\newcommand{\XX}{{\mathcal X}}
\newcommand{\XC}{{\mathcal X\mathcal C}}
\newcommand{\YY}{{\mathcal Y}}

\renewcommand{\CC}{{\mathcal D}} 
\newcommand{\DD}{{\mathcal C}}   
\renewcommand{\PP}{{\mathbb P}}
\renewcommand{\QQ}{{\mathbb Q}}

\newcommand{\TT}{{\mathcal T}}
\newcommand{\esssup}{{\mathrm {esssup}}}

\newcommand{\Ynn}{(Y^{(n)})_{n\in\N}}
\newcommand{\conv}{\operatorname {conv}}
\newcommand{\ep}{\varepsilon}

\newcommand{\ra}{\rightarrow}

\newcommand{\hy}{\hat{Y}}

\renewcommand{\ep}{\varepsilon}

\newcommand{\ind}[1]{ {\mathbf 1}_{\set{#1}}}
\newcommand{\inds}[1]{ {\mathbf 1}_{{#1}}}

\newcommand{\lf}{L^0_+(\FF)}

\newcommand{\ccp}{[\CC|\GG]^{\circ}}
\newcommand{\xpp}[1]{[#1|\GG]^{\circ}}
\newcommand{\fc}[2]{[#1|#2]^{\circ}}
\newcommand{\fcc}[2]{[#1|#2]^{\circ\circ}}
\newcommand{\ccpp}{[\CC|\GG]^{\circ\circ}}
\newcommand{\ccpccp}{\xpp{\ccp}}
\newcommand{\blp}{{\mathbf B}_+(\GG)}
\newcommand{\eg}[1]{\EE[#1|\GG]}
\newcommand{\et}[1]{\EE[#1|\FF_t]}
\newcommand{\es}[1]{\EE[#1|\FF_s]}
\newcommand{\epp}[1]{\EE[#1|\FF_p]}
\newcommand{\one}{{\mathbf 1}}

\newcommand{\oll}[1]{\overline{\overline{#1}}}
\newcommand{\al}{\alpha}
\newcommand{\hh}{\hat{h}}

\newcommand{\hm}{h_{max}}

\newcommand{\DP}{\DD^{\times}}

\newcommand{\tmk}{t^{(m)}_k}

\newcommand{\CP}{\CC^{\times}}
\newcommand{\CPP}{\CC^{\times\times}}

\begin{document}
\setlinespacing{1.2}
  \begin{center}%
    {\LARGE \bf A filtered version of the Bipolar Theorem of Brannath and
Schachermayer
\par}%

    \vskip 3em%
    {\large
     \bigskip
         Gordan \v Zitkovi\' c\footnote{Department of Statistics, Columbia
University, New York, NY}\\
Department of Statistics \\ 618
Mathematics Building \\ Columbia
University \\ New York, NY 10027\\
Phone:\  (212) 854-3652\\ \texttt{ gordanz@stat.columbia.edu}

} \vskip 1.5em    {\large January 4, 2001 \par}

  \end{center}

\begin{abstract}

\small\normalsize We extend the Bipolar Theorem of
Brannath and Scha\-cher\-ma\-yer (1999) to the space
of nonnegative c\` adl\` ag supermartingales on a
filtered probability space. We formulate the notion of
fork-convexity as an analogue to convexity in this
setting. As an intermediate step in the proof of our main
result we establish a conditional version of the Bipolar
theorem. In an application to mathematical finance we
describe the structure of the set of dual processes of
the utility maximization problem of Kramkov and 
Schachermayer (1999) and give a
budget-constraint characterization of admissible
consumption processes in an incomplete semimartingale
market.
 \end{abstract}

\


\section{INTRODUCTION}\
 
The classical Bipolar Theorem of functional analysis states
that the bipolar ${\mathcal D}^{\circ\circ}$ of a subset $\mathcal D$ of a
locally convex vector
space is the smallest closed, balanced and convex set containing $\mathcal
D$.
The locally convex structure of the underlying space is of great
 importance since the proof
relies heavily on the Hahn-Banach Theorem. In their
recent article, \cite{bipolar} exploit the order
structure of $L^0_+(\Omega,\FF,\PP)$ - (the space of all
nonnegative measurable functions on the probability space
$(\Omega,\FF,\PP)$ equipped with the topology of
convergence in measure) - to obtain an extension of the
Bipolar Theorem to this (generally not locally convex)
space. Indeed, if $\PP$ is a diffuse measure, the
topological dual of $L_0$ reduces to $\set{0}$ (see e.g.
\cite{kapr}, Theorem 2.2). Brannath and Schachermayer
consider a dual pair of convex cones $<L^0_+,L^0_+>$ with
the scalar product $<f,g>\mapsto \EE[fg]$ taking values
in $[0,\infty]$ and successfully identify the bipolar of
a subset $\CC$ of $L^+_0$ as the smallest convex, closed
in probability and solid set containing $\CC$. The
motivation for this extension comes from mathematical
finance, where it is customary to consider the natural
duality between the set of attainable contingent claims
and a variant of the set of all equivalent local
martingale measures.  For the problem of maximizing the
utility of the terminal wealth in  general incomplete
semimartingale securities market model, the set of all
(Radon-Nikodym densities of) equivalent local martingale
measures turns out to be too small - in terms of
closedness and compactness properties. The appropriate
enlargement, as described in \cite{4}, is obtained by
passing to the bipolar. This is where an operative
description - provided by the Bipolar Theorem for Subsets
of $L^0_+$ - is a sine qua non.

Inspired by, and heavily relying on, the result of
Brannath and Schachermayer, we decided to go one step
further and derive an analogue of the Bipolar Theorem for
sets of stochastic processes.  Additional motivation came
from mathematical finance - from an attempt to
characterize the optimal intratemporal consumption policy
for an investor in an incomplete semimartingale market.
Here it is not enough to study the relationship between
equivalent local martingale measures and attainable
contingent claims.  The time-dependent nature of the
problem forces us to consider the whole wealth process
and the corresponding dual "density processes" of
equivalent local martingale measures. Also, the
enlargement necessary to rectify the lack of closedness
and compactness properties of the set of all density
processes \text{(see \cite{4})} must take place in a
considerably more 'hostile' environment - the set of
nonnegative adapted stochastic processes. Specifically,
for a set of nonnegative c\` adl\` ag processes
${\mathcal X}$ defined in terms of stochastic integrals
with respect to a fixed semimartingale, the set $\YY^e$
of density processes corresponds to all strictly positive
c\` adl\` ag martingales $Y$ with $Y_0=1$ such that $(Y_t
X_t)_{t\in [0,T]}$ is a local martingale for all
$X\in\XX$. The enlargement (as proposed in \cite{4})
$\YY$ of $\YY^e$ consists of all nonnegative c\` adl\` ag
supermartingales $Y$ with $Y_0\leq 1$ such that $(Y_t
X_t)_{t\in [0,T]}$ is a supermartingale for each
$X\in{\mathcal X}$.

In this paper we abstract the important properties of
such an enlargement and phrase it in terms of a suitably
defined notion of the polar. In the manner of
\cite{bipolar} we put the set of all nonnegative adapted
c\` adl\` ag processes in duality with itself. However,
this time the scalar product is no longer a numerical
function anymore and it takes values in a suitably chosen
quotient space of the space of nonnegative stochastic
processes.

For our analysis we focus on sets  of nonnegative
supermartingales endowed with mild additional properties.
These properties are analogous to those of the set of all
density processes of equivalent local martingale
measures. In this context the new notion of
fork-convexity turns out to be the right analogue for the
concept of convexity in the classical case. We identify
the bipolar of a set of supermartingales as its
fork-convex, solid and closed hull, with notions of
solidity and closedness suitably defined. As a
by-product, we also obtain a conditional version of the
Bipolar Theorem which is, at least to the author, an
interesting result in its own sake. We then apply the
obtained results to describe the structure of the
enlarged set $\YY$ of dual density processes for the
problem of optimal consumption. For this case  we give a
simple budget-constraint characterization of all
admissible consumption densities. The results of this
paper can also be successfully applied to the problem of
optimal consumption in an incomplete semimartingale
market (see \cite{gz2}).

The paper is divided into 4 sections. Section 1 is the
introduction. In Section 2
 we present
the setting and state main theorems. Section 3 contains
the proofs, and in Section 4 we discuss applications to
mathematical finance.

\section{PRELIMINARIES AND THE MAIN RESULT}

In \cite{bipolar}, the following environment is
introduced. Let $(\Omega,\FF,\PP)$ be a probability
space, and let $L^0(\FF)$ denote the set of all
(equivalence classes) of real valued $\FF$-measurable
functions defined on $\Omega$. $L^0(\FF)$ becomes a
topological vector space if we endow it with the topology
of convergence in measure. With $L_+^0(\FF)=\set{ f\geq
0\,:\, f\in L^0(\FF)}$ being the positive orthant of
$L^0(\FF)$, it is possible to define a `scalar product'
on $L^0_+(\FF)$ by setting
\[ <f,g>\mapsto \EE[fg]\in [0,\infty].\]
In this way $L_+^0(\FF)$ is placed in duality with
itself. In this setting, Brannath and Schachermayer give
the following:
\begin{defn} Let $\CC$ be a subset of $L_+^0$. The set
\[ \CC^{\circ}=\set{g\in L_+^0\,:\, <f,g>\leq 1\ \mbox{for all}\
  f\in\CC}\]  is called the {\bf polar}
of $\CC$. A subset $\CC$ of $L_+^0$ is called
\begin{description}
\item[a)] {\bf solid} if for $f\in \CC$ and $g\in L_+^0(\FF)$,\, $g\leq
f$
  a.s. implies $g\in\CC$
\item[b)] {\bf closed} if it is closed with respect to the topology of
  convergence in probability.
\end{description}
\end{defn}
The Bipolar Theorem for Subsets of $L_+^0$ is given in
the following:
 \begin{thm}[{\bf Brannath and Schachermayer$^{(1999)}$}] \label{bp} Let $\CC$ be a subset of
$L_+^0(\FF)$. The bipolar
  $(\CC)^{\circ\circ}= (\CC^{\circ})^{\circ}$
of $\CC$ is the smallest closed, solid and
convex subset of $L_+^0(\FF)$ containing $\CC$.
\end{thm}
In our setting we would like to derive a similar theorem
in the context of stochastic processes. We fix a filtered
probability space $(\Omega,\FF,(\FFt)_{t\in [0,T]},\PP)$,
where $T>0$ is the {\bf time horizon}, and assume that
$(\FFt)_{t\in [0,T]}$ satisfies the {\em usual
conditions} with $\FF_0$ being the completed trivial
$\sigma$-algebra. We also introduce the following
notation and terminology for certain classes of
stochastic processes:
\begin{enumerate}
\item A nonnegative adapted c\` adl\` ag stochastic process
we will call
 a {\bf positive process} and we denote the set of all
  positive processes by $\PO$.
\item $\SU$ will denote the set of all supermartingales in $\PO$, and
  $\SUb$ the set of all supermartingales $Y$ in $\PO$ such that
  $Y_0\leq 1$.
\item A subset $\CC$ of $\SUb$ is called {\bf far-reaching} if there
  is an element $Y\in \CC$ such that $Y_T>0$ a.s.
\item $\DE$ denotes the set of all nonincreasing processes $B$ in
  $\PO$ such that $B_0\leq 1$.
\end{enumerate}

\begin{defn} Let $\CC$ be a subset of $\PO$.
The {\bf (process)-polar} of $\CC$ is the set of all positive
processes $Y$, such that $XY=(X_tY_t)_{t\in[0,T]}$ is a
supermartingale with $(XY)_0\leq 1$ for all $X\in \CC$. \end{defn}
\begin{rem} There is a formal analogy between our definition of
the polar and that for random variables in \cite{bipolar}. To show
how  $\PO$ is placed in duality with itself we first have to
define a suitable range space for the scalar product. Let
$\preceq_p$ be a binary relation on $\PO$ defined by
\[X\preceq_p Y\ \mbox{iff}\ X_0\leq Y_0\ \mbox{and}\ X-Y \ \mbox{is a
  supermartingale}.\] Defined as it is, $\preceq_p$ is not a partial
order. However, if we set $\RR$ to be the quotient space obtained
from $\PO$ by identifying processes whose difference is a local
martingale null at $0$, the natural projection $\preceq$ of
$\preceq_p$ to $\RR$ will define a partial order on $\RR$. If we
denote by $\FE$ the natural projection from $\PO$ onto $\RR$,  we
can see that a polar of a subset $\CC$ of $\PO$ is given by
\[ \CC^{\times}=\set{Y\in\PO\,:\, \FE[XY]\preceq 1\ \mbox{ for all
$X\in\CC$}}.\] \end{rem}

Our next task is to define analogues of solidity,
closedness and convexity. It turns out that the right
substitute for solidity is the following concept,
multiplicative in nature. We recall that $\DE$ stands for
the set of all nonincreasing processes $B$ in
  $\PO$ such that $B_0\leq 1$.
\begin{defn} Let $\CC$ be a subset of $\PO$. $\CC$ is called {\bf
    (process) solid} if for each $Y\in \CC$ and each $B\in\DE$ we have
  $YB\in\CC$.
\end{defn}

To define the appropriate notion of closedness, we recall
the concept of Fatou-convergence. It is an analogue of
convergence a.s. in the context of c\` adl\` ag processes
and was used for example in \cite{6}, \cite{7} and \mbox{
\cite{8}.}

\begin{defn}\label{FatD1}
Let $(Y^{(n)})_{n\in\N}$ be a sequence of positive processes.
 We say that $(Y^{(n)})_{n\in\N}$ {\bf Fatou converges}
to a positive process $Y$ if there is a countable dense subset $\TT$ of
$[0,T]$ such that
\begin{eqnarray}
 Y_t&=&\liminf_{s\downarrow t,s\in\TT}\liminf_n Y^{(n)}_s
\label{FatDef}\\\nonumber &=&\limsup_{s\downarrow
t,s\in\TT}\limsup_n Y^{(n)}_s
\end{eqnarray} for all t. We interpret $(\ref{FatDef})$ to mean
$Y_t=\lim_n Y^{(n)}_t$ for $t=T$.
\end{defn}
Fatou convergence has a number of desirable properties, especially
when applied to sequences in $\SU$. The following proposition is an easy
consequence of the Fatou Lemma:
\begin{prop}\label{subfsub}
Let $\Ynn$ be a sequence in $\SU$, Fatou converging to a positive process
$Y$. Then $Y$ is in $\SU$ as well. If additionally, $Y^{(n)}\in\SUb$ for
all $n$, then so is $Y$.
\end{prop}

\begin{defn} Let $\CC$ be a subset of $\PO$. $\CC$ is called {\bf
    closed} if it is closed with respect to Fatou convergence.
\end{defn}

Finally, we define the concept of {\bf fork-convexity}
for subsets of $\SU$. We want to look at processes in
$\SU$ as built up of multiplicative increments. In order
to be able do this we have to make sure that these
increments are well-defined. We refer the reader to
(\cite{5}, Prop. II.3.4, page 66) for the proof of the
following proposition.
\begin{prop} \label{yor} If $X$ is a nonnegative right-continuous
supermartingale and
\[ T=\inf\set{t\,:\, X_t=0}\wedge\inf\set{t>0\,:\, X_{t-}=0},\]
then, for almost every $\omega\in\Omega$, $X_{\cdot}(\omega)$
vanishes on $[T(\omega),\infty)$. \end{prop}
This result, together with the convention that
$\frac{0}{0}=0$ (which we freely use throughout the paper) allows us to
define random variables of the form $\frac{Y_t}{Y_s}$ for
$Y\in\SU$ and $t\geq s$.

\begin{defn}
A subset $\CC$ of $\SU$ is called {\bf fork-convex} if
for any $s\in [0,T]$, any $h\in L^0_+(\FF_s)$ with $h\leq 1$ a.s. and any
$Y^{(1)},Y^{(2)},Y^{(3)}\in\CC$, the process $Y$, defined by
\[ Y_t=\left\{
\begin{array}
{cl} Y^{(1)}_t& t<s \\
Y^{(1)}_s\left(
h\frac{Y^{(2)}_t}{Y^{(2)}_s}+(1-h)\frac{Y^{(3)}_t}{Y^{(3)}_s}
\right) & t\geq s
\end{array}\right.
\]
belongs to $\CC$.
\end{defn}
\begin{rem} The motivation for the introduction of fork-convexity
comes from mathematical finance. It can easily be shown that the set of
density processes of equivalent local martingale measures for a
semimartingale $S$ is
fork-convex. By {\em density process} of a probability measure $\QQ$ equivalent
to
$\PP$ we intend the c\` adl\` ag version of the martingale $Y^{\QQ}_t=\et{\frac{d\QQ}{d\PP}}$.
We refer the reader to \cite{7} for the related concept of {\em
predictable convexity}.
\end{rem}

Now we can state the main result of this paper.
\begin{thm}\label{fbt}{\em [\bf Filtered Bipolar Theorem]\em} Let $\CC$
be a far-reaching subset of $\SUb$. The process bipolar
$\CC^{\times\times}=(\CC^{\times})^{\times}$ of $\CC$ is
the smallest closed, fork-convex and solid subset of
$\SUb$ containing $\CC$. \end{thm}

An important ingredient in the proof of our main result
is the {\bf Conditional Bipolar Theorem}. This
conditional version may be interpreted as the Filtered
Bipolar Theorem in the setting of a discrete two-element
time set. Before stating the theorem, we give the
necessary definitions.
\begin{defn} Let $({\Omega,\FF,\PP})$ be a probability space and let
$\GG$ be
a sub-$\sigma$-algebra of $\FF$. A subset $\CC$ of $L_+^0(\FF)$ is
called
$\GG$-{\bf convex} if for all $f,g\in\CC$ and every $h\in L_+^0(\GG)$ with
$h\leq 1$ a.s., we have $hf+(1-h)g\in\CC$.
\end{defn}
\begin{defn} Let $({\Omega,\FF,\PP})$ be a probability space and let
$\GG$ be a sub-$\sigma$-algebra of $\FF$. For a subset $\CC$ of
$L_+^0(\FF)$ the set
\[ \ccp=\set{ g\in\lf\,:\, \EE[fg|\GG]\leq 1, \mbox{ a.s.}\ \mbox{for all}\
\, f\in \CC}.\]
is called the {\bf conditional polar of} $\CC$ {\bf with respect to}
$\GG$.
\end{defn}

We also recall the definition of boundedness in probability.
\begin{defn} A subset $\CC\in\lf$ is said to be {\bf bounded in
probability} if for each $\ep>0$ there is an $M>0$ such that
\[ \PP[f>M]<\ep\ \mbox{for all}\ f\in\CC.\]
\end{defn}
\begin{rem} We will use the following easy consequence of boundedness in probability:
Let $(f_n)_{n\in\N}$ be a sequence in $L^0_+$ converging a.s. to a random variable
$f$ with values in $[0,\infty]$. If $(f_n)_{n\in\N}$ is bounded in probability, then
so is $\set{f}\cup\set{f_n\,:\,n\in\N}$ and $f<\infty$ a.s.
\end{rem}
\begin{thm}\label{cbt}{\em [\bf  Conditional Bipolar Theorem]\em} Let
$\CC$ be a subset of $\lf$ which is bounded in probability. Then $\ccpp=\ccpccp$
is  the smallest $\GG$-convex,
solid and closed subset of $\lf$ containing $\CC$.
\end{thm}
\section{PROOFS OF THE THEOREMS}

We start this section by proving the Conditional Bipolar
Theorem via a number of auxiliary lemmata. $\blp$ denotes
the set of all nonnegative $\GG$-measurable functions on
$\Omega$ with expectation less then or equal to $1$. In
other words, $\blp$ is the intersection of $\lf$ with the
unit ball of $L^1(\GG)$.

\begin{lem}\label{jc} For $\CC\subseteq\lf$, $\ccp$ is solid,
$\GG$-convex and closed.
\end{lem}
\begin{proof}
Let $(g_n)_{n\in\N}$ be a sequence in $\ccp$ converging in probability to
$g\in\lf$. By passing to a subsequence we can assume $g_n\ra g$
a.s. By the Conditional Fatou Lemma, for every $f\in \CC$,
\[ \eg{fg}=\eg{\lim_n fg_n}\leq \liminf_n \eg{fg_n}\leq 1\
\mbox{a.s.},\] so $g\in\ccp$, i.e. $\ccp$ is closed.
$\GG$-convexity and solidity follow easily from the definition.
\end{proof}
 For $f\in \lf$ and $\CC\subseteq\lf$ we put
$f\CC=\set{fg\,:\, g\in\CC}$ and
$\frac{1}{f}\CC=\set{h\in\lf\,:\, hf\in\CC}$.
\begin{lem}\label{stbyci}
Let $\CC\subseteq\lf$. Then
\[ \ccp=\bigcap_{l\in\blp} \frac{1}{l} {\CC}^{\circ},\] and
\[ \ccpp=
\bigcap_{l\in\blp}\frac{1}{l}\overline{\overline{(\bigcup_{k\in\blp}
k\CC)}}\]
where $\CC^{\circ}=\set{g\in\lf\,:\, \EE[fg]\leq 1\ \forall\, f\in
\CC}$ is the (unconditional) polar of $\CC$ and $\overline{\overline{(\ \ \ )}}$
denotes the convex,
solid and closed hull.
\end{lem}
\begin{proof}
We observe that for $Z\in\lf$, $\eg{Z}\leq 1$ if and only
if $\EE[Zl]\leq 1$ for all $l\in\blp$ (the easy proof is
left to the reader). Then
\begin{eqnarray*}
\ccp&=&\set{g\in\lf\,:\, \eg{fg}\leq 1,\,\ \forall f\in \CC}\\ &=&
\set{g\in\lf\,:\, \EE[{fgl}]\leq 1,\,\ \forall f\in \CC,\,\ \forall\,\,
l\in\blp}\\ &=& \bigcap_{l\in\blp} \set{g\in\lf\,:\, \EE[fgl]\leq
1,\ \forall f\in \CC}\\&=& \bigcap_{l\in\blp} \frac{1}{l}\CC^\circ.
\end{eqnarray*}
IF we reiterate the same procedure and use the following
simple relations
\[ (k\CC)^\circ=\frac{1}{k} \CC^\circ\ \mbox{for any}\ k\in\lf,\ \mbox{and
any}\  \CC\subseteq\lf\]
and
\[ (\bigcup_{\alpha\in I} \CC_\alpha)^\circ= \bigcap_{\alpha\in I}
\CC_\alpha^\circ\ \mbox{ for any
family}\ (\CC_\alpha)_{\alpha\in I} \ \mbox{of subsets of}\ \lf,\]
we obtain \begin{eqnarray*}
\ccpp&=&\bigcap_{l\in\blp} \frac{1}{l} (\ccp)^{\circ} =
\bigcap_{l\in\blp} \frac{1}{l} (\bigcap_{k\in\blp} \frac{1}{k}
\CC^\circ) ^{\circ}\\ &=& \bigcap_{l\in\blp} \frac{1}{l}
(\bigcup_{k\in\blp} k\CC)^{\circ\circ} =\bigcap_{l\in\blp}
\frac{1}{l} \oll{(\bigcup_{k\in\blp} k\CC)},\end{eqnarray*} by the
(unconditional) Bipolar Theorem \ref{bp}.
\end{proof}
\begin{lem}
\label{secondpolar}
Let $\CC\subseteq\lf$ be bounded in probability. Then
\[\oll{(\bigcup_{k\in\blp} k\CC)}=\bigcup_{k\in\blp}
k\,\oll{\CC}^{\GG},\] and
\[ \ccpp=
\bigcap_{l\in\blp}\frac{1}{l}(\bigcup_{k\in\blp}
k\,\oll{\CC}^{\GG})\]
 where $\oll{(\ \ )}^{\GG}$ denotes the $\GG$-convex,
solid and closed hull.
\end{lem}
\begin{rem} \label{17} Lemma \ref{secondpolar} can be restated as follows:
\em{ For each $f\in\ccpp$ and each $l\in\blp$ there are $h\in
\oll{\CC}^{\GG} $ and $k\in\blp$ such that $fl=hk$.\em}
\end{rem}
\begin{proof}

From Lemma \ref{jc} and  that
$\xpp{\oll{\CC}^{\GG}}=\ccp$, we can assume without loss
of generality that $\CC$ is already $\GG$-convex, solid
and closed, because taking a $\GG$-convex, solid and
closed hull preserves boundedness in probability. Let
$\overline{(\ \ \ )}$ denote the closure with respect to
convergence in probability. We only need to prove that \[
\overline{(\bigcup_{k\in\blp} k\CC)}\subseteq
\bigcup_{k\in\blp} k \CC,\]  since $\cup_{k\in\blp} k\CC$
is a convex and solid subset of $
\overline{(\bigcup_{k\in\blp}k\CC)}$.
 Let $f\in\overline{(\cup_{k\in\blp} k\CC)}$. Then, there is a sequence
 $f_n$ converging to $f$ in
probability, and each $f_n$ is of the form
$l_nh_n$ for some $l_n\in\blp$ and $h_n\in \CC$. By passing to a
subsequence, we can assume that $f_n\ra f$ a.s. The sequence $l_n$
is bounded in $L^1$, so by Komlos' Theorem (see \cite{S86} and references
therein
for a
good exposition and generalizations) there is a sequence
of convex combinations
\[ k_n\in\conv{ (l_n,l_{n+1},\ldots)}\subseteq\blp\] converging to a
random variable $l'$ a.s. By Fatou Lemma, $\EE[l']\leq 1$ so $l'\in\blp$. If
$k_n$ is of the form $\sum_{j=n}^{m_n} \al^n_j l_j$, we define (recalling
that $\frac{0}{0}=0$)
\[ k^n_j=\frac{\al^n_j l_j}{k_n}\ \mbox{and}\
\hh_n=\sum_{j=n}^{m_n} k^n_j h_j.\]  By $\GG$-convexity of $\CC$,
$\hh_n\in \CC$ because
$\sum_{j=n}^{m_n} k_j^n=1$. If we redefine $\hh_n$ by putting
$\hh_n=0$ on $\set{f=0}$ we still have $\hh_n\in \CC$, and  the
relation
\[ f=\lim_n\sum_{j=n}^{m_n} \al_j^n l_j h_j=\lim_n k_n \hh_n\]
allows us to conclude that on $\set{f>0}\cap\set{l'=0}$, $\hh_n$ must converge
to
$+\infty$ a.s. However, $\CC$ is bounded in probability so we must have
$\PP[\set{f>0}\cap\set{l'=0}]=0$. It is now clear that there is a finite random
variable $\hh\in \overline{\CC}=\CC$ such that $\hh_n\ra \hh$ a.s.
Therefore

\[ f=l'\hh\in\bigcup_{k\in\blp} k\CC.\]
\end{proof}
\begin{proof}[Proof ({\bf Conditional Bipolar Theorem \ref{cbt}})]
Without loss of generality we assume $\CC$ is already
closed, solid, $\GG$-convex and bounded in probability.

For $f\in\ccpp$ we define
\[ H^f=\set{ h\in \CC\,:\, \set{f=0}\subseteq\set{h=0}\ \mbox{and}\
f\eg{h}=h\eg{f}}.\] By Remark \ref{17} we can choose
$h'\in\CC$ and $g\in\blp$ such that $f=gh'$. Then let
$h=h'\ind{f=0}$ and obtain
 $h\eg{f}=gh\eg{h}=f\eg{h},$ so $h\in H^f$ implying that $H^f$ is not empty.

In order to prove the theorem we need to show that $f$ is
dominated by an element of $H^f$. We first show that
$\hm=\esssup H^f\in H^f$.
 For $ h_1, h_2 \in H^f$ define
$C=\set{\eg{h_1}>\eg{h_2}}$. From the defining property
of $H^f$ we have
\begin{eqnarray*} (h_1\vee h_2)\eg{f}&=&(\eg{h_1}\vee\eg{h_2})f=
(\inds{C}\eg{h_1}+\inds{C^c}\eg{h_2})f\\&=&
(\inds{C}h_1+\inds{C^c}h_2)\eg{f}.\end{eqnarray*} As
$h_1=h_2=0$ on $\set{\eg{f}=0}$ we immediately conclude
that $h_1\vee h_2= \inds{C}h_1+\inds{C^c}h_2$ and
$h_1\vee h_2\in\CC$ by $\GG$-convexity. We proceed
further and note that
\begin{eqnarray*} \eg{h_1\vee
h_2}f&=&\eg{h_1\inds{C}+h_2\inds{C^c}}f=(h_1\inds{C}+h_2\inds{C^c})\eg{f}
\\&=&(h_1\vee h_2) \eg{f},\end{eqnarray*}
so $h_1\vee h_2\in H^f$ which proves that $H^f$ is closed under pairwise
maximization. By Theorem A.3. in \cite{KS} , $\hm$ can be
written as $\hm=\lim_n h_n$, where $h_n$ is a nondecreasing
sequence in $H^f$. Boundedness in probability of $\CC$
and the monotone convergence theorem imply that $\hm<\infty$ a.s.
and ${\hm}\eg{f}=f\eg{\hm}$. Finally,  $\hm\in
H^f$ because $\CC$ is closed.

To see that $f$ is dominated by $\hm$, we define
$A=\set{\eg{f}>\eg{\hm}}$ and
 assume $\PP[A]>0$. With $l=\inds{A}/\PP[A]\in\blp$, by Remark \ref{17},
there
are $h\in\CC$ and $k\in\blp$ such that $lf=kh$. Without loss of generality we
may assume that $h\in H^f$. As $k\hm\geq kh$, through conditioning upon
 $\GG$ we have that $k\eg{\hm}\geq l\eg{f}$ so
 \[ l\eg{f}=\inds{A}l\eg{f}\leq \inds{A}k\eg{\hm}\leq \inds{A} k \eg{f}.\]
 The fact that $\eg{f}$ is strictly positive on $A$ leads us to conclude that
$k\geq l$ on $A$. As $k\in\blp$ and $\EE[l]=1$ we must have $k=l$ and,
consequently, $f\leq \hm$ on $A$. Taking conditional expectation we get
$\PP[A]=0$.

 We have shown that $\eg{f}\leq \eg{\hm}$ and the definition of $H^f$ immediately
yields
$f\leq \hm$. In other words, $f$ is dominated by an element of $\CC$, implying
that $f\in\CC$ by solidity. Thus $\ccpp\subseteq \CC$. The converse inclusion is
obvious.
\end{proof}

We may now proceed gradually to the proof of our main
result. For a process $Y\in\SU$ and $t>s$, we denote by
$\frac{Y_t}{Y_s}$ the multiplicative increment   and we
define $\Delta_{t,t} Y=1$.
\begin{prop} \label{one} Let $\DD\subseteq \PO$ with $\one\in \DD$, where $\one$
denotes the constant process equal to $1$. Then $\DP$ is a closed,
solid and fork-convex subset of $\SUb$.
\end{prop}
\begin{proof}
Since $\one$ is in $\DD$, obviously $\DP\subseteq \SUb$. Let $(Y^n)_{n\in\N}$ be a
sequence in $\DP$, Fatou-converging to some $Y\in\SUb$ and let $X\in \DD$.  Since all
$XY^n$ are in $\SUb$ and  $XY^n$ Fatou-converges to $XY$, $XY$
is in $\SUb$ as well by Proposition \ref{subfsub}. This proves the
closedness of $\DP$.
Let $Y\in \DP$ and $B\in\DE$. Then , for all $X\in \DD$,
$BXY$ is a supermartingale because $XY$ is one. Finally,
let $Y^{(1)},Y^{(2)},Y^{(3)}\in \DP$, $t_0\in [0,T]$ and $h\in
L^0_+(\FF_{t_0})$
with $0\leq h\leq 1$ and let a process Y be defined by

\[Y_t=\left\{
\begin{array}
{cl} Y^{(1)}_t& t<t_0\\
Y^{(1)}_{t_0}\left(
h\frac{Y^{(2)}_{t}}{Y^{(2)}_{t_0}}+(1-h)\frac{Y^{(3)}_{t}}{Y^{(3)}_{t_0}}
\right) & t\geq t_0
\end{array}\right.
.\] We want to prove that $\EE[X_tY_t|\FF_s]\leq X_sY_s$,
for all $t>s$ and all $X\in\DD$. To do so, we only
consider the case $s=t_0,t=T$. The other cases can be
dealt with analogously. By definition of $Y$ and the fact
that $XY^{(1)}$, $XY^{(2)}$ and $XY^{(3)}$ are
supermartingales,
\begin{eqnarray*}
\EE[Y_tX_t|\FF_{t_0}]&=& \EE[Y^{(1)}_{t_0} h \Delta_{T,t_0}Y^{(2)}
X_t|\FF_{t_0}]+ \EE[Y^{(1)}_{t_0} (1-h) \Delta_{T,t_0}Y^{(3)}
X_t|\FF_{t_0}]\\ &\leq&
Y^{(1)}_{t_0}(hX_{t_0}+(1-h)X_{t_0})=Y^{(1)}_{t_0}X_{t_0}=Y_{t_0}X_{t_0}.
\end{eqnarray*}
$XY$ is, therefore, a supermartingale so $Y\in\DP$.
\end{proof}

The proof of the following lemma was inspired by
techniques in Kram- $\text{kov}^{(1996)}$.
\begin{lem} Suppose $\CC$ is a fork-convex far-reaching subset of $\SU$.
 Let $t_1\leq t_2\in [0,T]$
and let $g\in L^1_+(\FF_{t_2})$ be such that  $\esssup_{Y\in \CC}
\EE[g\Delta_{t_2,t_1}Y|\FF_{t_1}]\leq 1$ a.s. If we define the
process $X$ by
\[ X_t=\left\{\begin{array}{cl} 1 & t<t_1\\ \esssup_{Y\in \CC}
\EE[g\Delta_{t_2,t}Y|\FF_t]& t\in [t_1,t_2) \\ g& t\in
[t_2,T],\end{array}\right. \] then $XY$ is a supermartingale for
each $Y\in \CC$ and $X$ permits a c\`{a}dl\`{a}g modification. \label{cadl}
\end{lem}
\begin{proof} Without any loss of generality we assume
$t_2=T$. First we will prove that, for $t\geq t_1$, there
is a sequence $(Y^n)_{n\in\N}\in \CC$ such that
\begin{equation}\label{jen}
\EE[g\Delta_{T,t}Y^n|\FF_t]\nearrow X_t,
\end{equation} as $n\to\infty$. By Theorem A.3.  in \cite{KS} it is enough
to prove that
the set $ \set{ \EE[g\Delta_{T,t} Y|\FF_t]\,:\, Y\in \CC} $ is
closed under pairwise maximization. Let $Y^{(1)},Y^{(2)}$ be in $\CC$. Put
\[ h=\ind{\et{g\Delta_{T,t} Y^{(1)}}\geq \et{g\Delta_{T,t} Y^{(2)}}}.\]
Then for the process $Y^{max}$ defined by
\[Y^{max}_s=\left\{
\begin{array}
{cl} Y^{(1)}_s& s<t\\
Y^{(1)}_t\left(
h\frac{Y^{(1)}_{s}}{Y^{(1)}_t}+(1-h)\frac{Y^{(2)}_{s}}{Y^{(2)}_t}
\right) & s\geq t
\end{array}\right.
\] fork-convexity implies that $Y^{max}\in \CC$ and
\[ \et{g\Delta_{T,t} Y^{max}}=\et{g\Delta_{T,t} Y^{(1)}}\vee
\et{g\Delta_{T,t} Y^{(2)}}.\] Fix $t_1\leq s\leq t\leq t_2$ and a sequence
$(Y^n)_{n\in\N}$ such
that (\ref{jen}) holds. By the Monotone Convergence Theorem, for each
$Y\in\CC$, we have
\begin{eqnarray*} \es{Y_t X_t}&=& \es{\lim_n Y_t
\et{g\Delta_{T,t} Y^n}}=\lim_n \es{ Y_t g\Delta_{T,t} Y^n}\\ &=&
Y_s\lim_n\es{ g\Delta_{T,t}Y^n\Delta_{t,s} Y}.\end{eqnarray*}
By fork-convexity, $\Delta_{T,t}Y^n\Delta_{t,s} Y_s$ is equal to
$\Delta_{T,s}\tilde{Y}$ for some $\tilde{Y}\in \CC$, so
\[\lim_n\es{ g\Delta_{T,t}Y^n\Delta_{t,s} Y_s}\leq\esssup_{Y\in
\CC}\es{g\Delta_{T,s} Y}=X_s.\] Therefore $\es{Y_t X_t}\leq Y_s
X_s$ and so $XY$ is a supermartingale on $[t_1,t_2)$ for all $Y\in
\CC$. Because of the condition $\esssup_{Y\in \CC}
\EE[g\Delta_{t_2,t_1}Y|\FF_{t_1}]\leq 1$ a.s., $XY$ is a
supermartingale on the whole interval $[0,T]$.

To prove that $X$ has a c\'{a}dl\`{a}g version, we will
first prove that $S=X\hy$ has one, where $\hy$ is an
element of $\CC$ such that $\hy_T>0$ a.s. The process $S$
is a supermartingale so it is enough to prove that
$t\mapsto \EE[S_t]$ is right-continuous (see \cite{5},
Theorem II.2.9, page 61). We fix $p\in [0,T]$ and a
sequence $(p_n)_{n\in\N}$ such that $p_n\searrow p$, and
consider the only non-trivial case  - namely, when $p\in
[t_1,t_2)$. Let $Y^n$ be a sequence in $\CC$ such that \[
\hy_p \epp{g\Delta_{T,p} Y^n}\nearrow S_p.\] For $\ep>0$
there is an $n\in\N$ such that $\EE[\hy_p
g\Delta_{T,p}Y^n]>\EE[S_p]-\ep$. By right-continuity of
processes in $\CC$ and Fatou Lemma
\begin{eqnarray*}
\EE[\hy_pX\Delta_{T,p} Y^n] &=& \EE[\lim_k (\hy_{p_k} g
\Delta_{T,p_k} Y^n)] \\ &\leq& \liminf_k\EE[ \hy_{p_k} g
\Delta_{T,p_k} Y^n]\\&\leq&\lim_k
\EE[S_{p_k}].\end{eqnarray*} Now, $\lim_k
\EE[S_{p_k}]\geq \EE[S_p]-\ep$ for all $\ep>0$, so
$t\mapsto \EE[S_t]$ is right continuous. Therefore
$\hy_tX_t$ has a c\`{a}dl\`{a}g modification $P_t$. Since
$\hy_T>0$, $\set{(t,\omega)\,:\, \hy_t(\omega)=0\ \mbox{
or} \ \hy_{t-}(\omega)=0}$ is an evanescent set so we
conclude that $\frac{P_t}{\hy_t}$ is a c\`{a}dl\`{a}g
modification of $X_t$.
\end{proof}
\begin{rem}  For $t_1<t_2\in [0,T]$ and
$\DD\subseteq \PO$  we put $\DD_{t_2,t_1}=\set{\Delta_{t_2,t_1}
X:X\in \DD}$ whenever it is well-defined.
We note that if $X\in\PO$ and $\hy\in \SU$ with $\hy_T>0$ a.s. such that $X\hy$ is
a supermartingale, then $X$ has the following property (inherited from $\hy X$): if
$t_1<t_2\in [0,T]$ and $X_{t_1}=0$ on $A\in\FF_{t_1}$, then
$X_{t_2}=0$ on $A$ as well. Therefore,
$\Delta_{t_2,t_1} X$ is well-defined, if $\DD$ is a polar of a far-reaching
subset of $\SU$.
\end{rem}
\begin{lem} Let $\CC\subseteq \SUb$ be a fork-convex, solid and far-reaching set.
Then, for all $t_1<t_2\in
[0,T]$, $\CC_{t_2,t_1}$ is solid, convex and
\[ \fc{\CC_{t_2,t_1}}{\FF_{t_1}}=(\CP)_{t_2,t_1},\]
 where all random variables in the definitions of solidity and conditional polar are
 assumed to be $\FF_{t_2}$-measurable.\end{lem}
\begin{proof} The solidity and convexity of $\CC_{t_2,t_1}$ follow from
the solidity and fork-convexity of $\CC$. By the previous remark,
$\CP_{t_2,t_1}$ is well defined. Let $g\in
\fc{\CC_{t_2,t_1}}{\FF_{t_1}}\subseteq L^+_0(\FF_{t_2})$.
Then \[\esssup_{Y\in \CC}
\EE[g\Delta_{t_2,t_1} Y| \FF_{t_1}]\leq 1\] so, by Lemma
\ref{cadl} the c\`{a}dl\`{a}g version of the process
\[ X_t=\left\{\begin{array}{cl} 1 & t<t_1\\ \esssup_{Y\in \CC}
\EE[g\Delta_{t_2,t}Y|\FF_t]& t\in [t_1,t_2) \\ g& t\in
[t_2,T],\end{array}\right. \] is in $\CP$. Moreover,
$\frac{X_{t_2}}{X_{t_1}}\geq g$, so, by solidity, $g\in
(\CP)_{t_2,t_1}$. Conversely, let $h\in \CP_{t_2,t_1}$ be of the from
$h=\frac{X'_{t_2}}{X'_{t_1}}$. By definition,
$\EE[X'_{t_2}Y_{t_2}|\FF_{t_1}]\leq X'_{t_1}Y_{t_1}$, so
$\EE[h\Delta_{t_2,t_1} Y| \FF_{t_1}]\leq 1$ for all $Y\in \CC$.
Therefore $h\in \fc{\CC_{t_2,t_1}}{\FF_{t_1}}$.
\end{proof}
\begin{lem} Let $\CC$ be a far-reaching subset of $\SUb$.
 Pick $0=t_0< t_1<t_2<\ldots < t_m\leq T$ and  $Y\in \CPP$.
Then there is a sequence $Y^n$ of elements of the solid and
fork-convex hull $\oll{\CC}$ of $\CC$ such that $\lim_n Y^n_{t_k}=Y_{t_k}$ a.s.,
for $k=0,\ldots,m$.
\end{lem}
\begin{proof}
 Let $0\leq t_1<t_2\leq T$.
The set $\set{Y_0\,:\,Y\in \oll{\CC}}$ is a subinterval of
$[0,\infty)$ containing $0$ and hence the bipolar $\set{Y_0\,:\,
Y\in \CPP}$ is just the closure of this interval. Therefore  there
is a sequence $Y^{(0,n)}\in \CC$, $n\in\N$ such that $\lim_n
Y^{(0,n)}_0=Y_0$. By the Bipolar Theorem \ref{bp} and the previous
lemma,
\[\overline{(\oll{\CC}_{t_1,0})}=(\oll{\CC}_{t_1,0})^{\circ\circ}=(\CC_{t_1,0})^{\circ\circ}
=\fcc{\CC_{t_1,0}}{\FF_{0}}=\CPP_{t_1,0},\] where  \,$
\overline{(\ \ )}$ denotes  closure with respect to the
topology of convergence in probability and we take all
random variables in the definitions of polars involved to
be $\FF_{t_1}$-measurable. We conclude there is a
sequence $(Y^{(1,n)})_{n\in \N}\in\oll{\CC}$ such that
$\Delta_{t_1,0} Y^{(1,n)}\to \Delta_{t_1,0} Y$ when
$n\to\infty$. Similarly,
\[\overline{(\oll{\CC}_{t_2,t_1})}=(\CPP)_{t_2,t_1},\]
by the Conditional Bipolar Theorem \ref{cbt}, so there is a
sequence $(Y^{(2,n)})_{n\in \N}\in\oll{\CC}$ such that
$\Delta_{t_2,t_1} Y^{(2,n)}\to \Delta_{t_2,t_1} Y$ as
$n\to\infty$. We continue this procedure to construct sequences
$(Y^{(k,n)})_{n\in\N}\in\oll{\CC}$  such that
$\Delta_{t_{k},t_{k-1}} Y^{(k,n)}\to \Delta_{t_{k},t_{k-1}} Y$ as
$n\to\infty$ for $k=3,\ldots, m$.

By fork-convexity and solidity of $\oll{\CC}$, there is a sequence
$Y^{(n)}\in\oll{\CC}$ such that $Y^{(n)}_0=Y^{(0,n)}_0$ and
$\Delta_{t_{k},t_{k-1}} Y^{(n)}=\Delta_{t_k,t_{k-1}} Y^{(k,n)}$ so
$Y^{(n)}_{t_k}\to Y_{t_k}$ as $n\to\infty$ for $k=0,1,\ldots,m$.
\end{proof}
\begin{lem} Let $\CC$ be a far-reaching subset of $\SUb$.
 For each $Y\in\CPP$ there is a sequence $Y^{(n)}$ in the solid and fork-convex
 hull $\oll{\CC}$  of $\CC$ such that
$Y^{(n)}_q\to Y_q$ a.s., for all $q\in Dy$, where $Dy=\{
qT\,:\, q\ \mbox{is a dyadic rational in}$ \ $[0,1]\}$.
\end{lem}
\begin{proof}
Let $Y\in\CPP$. Define $\tmk=\frac{k}{2^m}T$ for
$m\in\N,k\in\set{0,1,\ldots,2^m}$. By the previous lemma,
for each $m$ we can find a sequence
$(Y^{(m,n)})_{n\in\N}$ such that \[ \lim_n
Y^{(m,n)}_{\tmk}=Y_{\tmk}\ \mbox{a.s.},\] for all $m,k$.
The sequence $(Y^{(m,n)})_{n\in\N}$ can be chosen in such
a way that for $m\in\N,k\in\set{0,1,\ldots,2^m},n\in\N$,
\[\PP[| Y^{(m,n)}_{\tmk}-Y_{\tmk}|>2^{-n}]<2^{-n}.\]
This will ensure that for the diagonal sequence
$Y^{(n)}=Y^{(n,n)}$, $Y^{(n)}_{q}\to Y_{q}$ a.s. for all $q\in
Dy$.
\end{proof}

\begin{proof}[Proof ({\bf of the Filtered Bipolar Theorem \ref{fbt}})]
Let $\CC'$ be the smallest solid, fork-convex and closed subset
of $\SUb$ containing $\CC$. By the previous lemma, for each $Y\in\CPP$ there
is a sequence $(Y^{(n)})_{n\in\N}$ such that $Y^{(n)}_q\to Y_q$ as $n\to\infty$ for
each $q\in Dy$. $Y$ is c\' adl\' ag so it follows from the definition that
$Y^{(n)}$ Fatou converges to $Y$, and, consequently, $Y\in\CC'$ so $\CPP\subseteq\CC'$.
Conversely, since $\one\in\CP$, Proposition  \ref{one} implies $\CC'\subseteq\CPP$,
thus proving the claim of the theorem.
 \end{proof}

\section{AN APPLICATION TO MATHEMATICAL FINANCE}

Let $S$ be a semimartingale taking values in $\R^d$ defined on
$(\Omega,\FF,(\FF_t)_{t\in [0,T]},$ $\PP)$.  We will interpret $S$
as the price process of $d$ risky assets in a securities market
with the time horizon $T$. By taking the constant process $1$ as
the num\' eraire we assume our prices are already discounted.

An agent with initial endowment $x$ investing in this
market chooses a predictable $S$-integrable process $H$
and an adapted nondecreasing c\`{a}dl\`{a}g process $C$
with $C_0=0$. The triple $(x,H,C)$ is called an {\bf
investment-consumption strategy}.  The process $H_t$ can
be interpreted as the amount of each asset held, and
$C_t$ is the cumulative amount spent on consumption prior
to time $t$.

With an investment-consumption
strategy we associate a process $X^{x,H,C}$ defined by
\begin{eqnarray}
X^{x,H,C}=x+\int_0^t H_u\, dS_u-C_t.
\end{eqnarray}
$X^{x,H,C}$ represents the value of the agent's current
holdings and is called {\bf the wealth process associated
with investment-consumption strategy} $(x,H,C)$ . An
investment-consumption strategy $(x,H,S)$ is called {\bf
admissible} if its wealth process remains nonnegative,
i.e. if $X^{x,H,C}_t\geq 0$ for all $t$. The set of all
wealth processes of admissible investment-consumption
strategies with an initial endowment of less than or
equal to $x$ will be denoted by $\XC(x)$. If for an
investment-consumption strategy $(x,H,C)$, we have
$C\equiv 0$, the pair $(x,H)$ is called a {\bf pure
investment strategy} and the set of all wealth processes
of admissible pure investment strategies with initial
endowment less than or equal to $x$ is denoted by
$\XX(x)$.

To have a realistic model of the market, we assume a
variant of the non-arbitrage property by postulating the
existence of a probability measure $\QQ$ on $\FF$,
equivalent to $\PP$, such that each $X\in\XX(1)$ is a
local martingale under $\QQ$. Any such measure $\QQ$ is
 called {\bf an equivalent local martingale measure}, and the set
of all such measures is denoted by $\MM$
 (  we
refer the reader to \cite{FT} and \cite{UNB} for an
in-depth analysis of the relation between existence of
equivalent local martingale measures and the
non-arbitrage properties). If $Y^{\QQ}$ is a
c\`{a}dl\`{a}g process of the form \[
Y^{\QQ}_t=\et{\frac{d\QQ}{d\PP}}\] for some $\QQ\in\MM$,
then $Y^{\QQ}$ is called a {\bf local martingale density}
and $\YY^e$ denotes the set of all such processes. The
Optional Decomposition Theorem (see \cite{9} for the
original result, \cite{6}, \cite{3}, \cite{7}, and
\cite{8} for more general versions) is the fundamental
tool in our analysis. In particular, in our setting,
Theorem 2.1. in \cite{6} states that a nonnegative
c\`{a}dl\`{a}g process $X$ with $X_0\leq x$ is in
$\XC(x)$ if and only if $X$ is a supermartingale under
each $\QQ\in\MM$. Similarly, $X$ is in $\X(x)$ if and
only if $X$ is a local martingale under each $\QQ\in\MM$.
\begin{rem} The Bayes rule for stochastic processes (see Lemma 3.5.3,
page 193. in \cite{KSB})  and the fact that
$x\XC(1)=\XC(x)$ imply that\label{r23}
$\XC(1)=(\YY^e)^{\times}$, so a nonnegative c\' adl\' ag
process $X$ is in $\XC(x)$ if and only if $X_0\leq x$ and
$XY^{\QQ}$ is a nonnegative c\' adl\' ag supermartingale
for all $Y^{\QQ}\in\YY^e$.\end{rem} Under certain duality
considerations (see Kramkov and
Scha\-cher\-ma\-yer$^{(1999)}$) it is necessary to
enlarge $\YY^e$ to
\[ \YY=\set{Y\in\PO\,:\, Y_0\leq 1, \mbox{ $YX$ is a supermartingale
for all $X\in\XX(1)$}}=(\XX(1))^{\times}\] in order to rectify poor
closedness properties of $\YY^e$.
The main result of this section describes the structure
of $\YY$, $\XX(1)$ and
$\XC(1)$: \
\bigskip
\begin{thm}\
\nopagebreak
\begin{description}\item{(a)}
$\YY=(\XC(1))^{\times}$, so  $Y(X-C)$ is a
supermartingale for all $X-C\in\XC(1)$.
\item{(b)} $\YY$ is closed, fork-convex and solid,
$\YY=(\YY^e)^{\times\times}$, and for each $Y\in\YY$ there is a
sequence $Y^{(n)}$ of elements in the solid hull of $\YY^e$ such
that $Y^{(n)}\to Y$ in the Fatou sense.
\item{(c)} $(\XX(1))^{\times\times}=(\XC(1))$.
\end{description}\label{myprop}
\end{thm}
\begin{proof}
\begin{description}
\item{(a)} By Remark \ref{r23}, it is sufficient to prove that
$\YY\subseteq (\XC(1))^{\times}$. Let $Y\in\YY$ and let
$Z_t=X_t-C_t\geq 0$ be an element of $\XC(1)$ with
$X\in\XX(1)$. Also, let  $C$ an nondecreasing, adapted,
c\`{a}dl\`{a}g process such that  $C_0=0$. We assume
$X_0=1$, take $s<t$ in $[0,T]$ and find $H$, a
predictable $S$-integrable process such that
$X_{\cdot}=1+\int_0^{\cdot} H_u\,dS_u$. We define a
process $K$ by putting $K_u=H_u$, for all $u$, on the set
$\set{X_s=C_s}$ and, for $u\leq s$, on $\set{X_s>C_s}$.
For $u>s$ on $\set{X_s>C_s}$ we set
\begin{eqnarray}
K_u=\frac{X_s}{X_s-C_s} H_u.
\end{eqnarray} so that $K$ is an $S$-integrable predictable
process.  To show that $(1,K)$ is an admissible pure
investment scheme, we that (with convention
$\frac{0}{0}=0$)
\begin{eqnarray}
\label{NewPredictable} 1+\int_0^v K_u\,dS_u=\left\{
\begin{array}{lll}
X_v & ,& v\leq s\\ X_s+\frac{X_s}{X_s-C_s} (X_v-X_s) &,& v>s
\end{array}\right.
\end{eqnarray}
because that $X_v=X_s$ for $v\geq s$ on $\set{X_s=C_s}$
by Proposition \ref{yor}, since $X-C$ is a
supermartingale under each $\QQ\in\MM$. Furthermore,
\begin{eqnarray}
X_s+\frac{X_s}{X_s-C_s} (X_v-X_s)&\geq&
X_s+\frac{X_s}{X_s-C_s} (X_v-C_v-(X_s-C_s))\nonumber
\\&=&X_s\frac{(X_v-C_v)}{X_s-C_s}\geq 0.
\end{eqnarray}
Therefore $(1,K)$ is indeed an admissible pure investment
scheme, and, by the definition of $\YY$, the process
$Y_v(1+\int_0^v K_u\,dS_u)$ is a supermartingale. Thus we
may write
\begin{eqnarray} \nonumber
Y_sX_s&=&Y_s\left(1+\int_0^s K_u\,dS_u\right)\geq \EE\left[Y_t\left(
1+\int_0^t
K_u\,dS_u\right)|\FF_s\right] \\ \nonumber
&=&\EE\left[Y_t\left(X_s+\int_{s+}^t
\frac{X_s}{X_s-C_s}H_u\,dS_u\right)|\FF_s\right]\\&=&\frac{X_s}{X_s-C_s}
\EE\left[Y_t\left(X_t-C_s\right)|\FF_s\right]\\&\geq&\frac{X_s}{X_s-C_s}
\EE\left[Y_t\left(X_t-C_t\right)|\FF_s\right].\nonumber
\end{eqnarray}
If we multiply both sides by $\frac{X_s-C_s}{X_s}$ we get
the desired supermartingale property.
\item{(b)} $\YY^e$ is obviously a subset of $\SUb$, and it
is far-reaching since for each $Y\in\YY^e$, there is a $\QQ\in\MM$
such that $Y_T=\frac{d\QQ}{d\PP}>0$ a.s. Since
$\YY=(\XC(1))^{\times}=(\YY^e)^{\times\times}$ and because of  the
Filtered Bipolar Theorem \ref{fbt}, $\YY$ is the smallest solid,
fork-convex and closed subset of $\SUb$ containing $\YY^e$. It is
easy to check that $\YY^e$ is fork-convex so from the proof of
Theorem \ref{fbt} we infer that for each $Y\in\YY$ we can find a
sequence $Y^{(n)}$ is the solid hull of $\YY^e$ such that
$Y^{(n)}$ Fatou-converges to $Y$.
\item{(c)} From (b) and Remark \ref{r23},
\[ \XC(1)\subseteq
(\XC(1))^{\times\times}=(\YY^e)^{\times\times\times}
=(\YY)^{\times}\subseteq(\YY^e)^{\times}= \XC(1),\] so
$\XC(1)=(\YY)^{\times}=(\XX(1))^{\times\times}$.
\end{description}
\end{proof}

As a corollary to this result we also give a simple
duality characterization of admissible consumption
processes. Let $\mu$ be a probability measure on the
Borel sets of $[0,T]$, diffuse on $[0,T)$, (i.e. the only
atom we allow is on $\set{T}$). A process $C_t$ will be
called an $x$-{\bf admissible consumption process} if
there is an admissible investment-consumption strategy of
the form $(x,H,C)$. If there is a progressively
measurable nonnegative process $c$ such that
$C_t=\int_0^t c(u)\,\mu(\du)$, then $C$ will be called an
{\bf absolutely continuous consumption process} and $c$
its {\bf consumption density}.
\begin{cor}\label{koro}
Let $x>0$. A nonnegative progressively measurable process $c$ is a
consumption density of an $x$-admissible absolutely continuous
consumption process if and only if
\begin{equation}\label{chr} \sup_{Y\in\YY}\EE\left[\int_0^T
Y_uc(u)\,\mu(\du)\right]\leq x.\end{equation}
\end{cor}
Before we give the proof, we need the following lemma:
\begin{lem}\label{linf}
If $\ynn$ is a sequence in $\SUb$, Fatou converging to $Y\in\SUb$,
then there is a countable set $K\subseteq [0,T)$ such that for
$t\in[0,T]\setminus K$, we have $Y_t=\liminf_n \ynn_t$ a.s.
\end{lem}
\begin{proof} Let $\TT$ be the countable dense subset of $[0,T]$
as in the definition of Fatou convergence for
$(\ynn)_{n\in\N}$. Putting $Y'_t=\liminf_n \ynn_t$, it is
easy to see that $Y'_t$ is a nonnegative supermartingale.
Let $K$ be the set of points of right-discontinuity of
the function $t\mapsto \EE[Y'_t]$. Since $Y'$ is a
supermartingale, $t\mapsto \EE[Y'_t]$ is nonincreasing so
$K$ is a countable subset of $[0,T)$. For $t\in
[0,T]\setminus K$, and $q_n\searrow t$, $q_k\in \TT$,
$(Y'_{q_k})_{k\in\N}$ is a backward supermartingale
bounded in $L^1$.  By the Backward Supermartingale
Convergence Theorem (see \cite{5}, Theorem II.2.3, page
58) and the definition of Fatou convergence, $Y'_{q_k}\to
Y_t$ a.s and in $L^1$. Therefore
\begin{equation} \label{fin1}
\EE[Y_t]=\EE[Y'_t].\end{equation} On the other hand, for
$t\in [0,T]\setminus K$,
\begin{equation}
\label{fin2}
Y_t=\et{Y_t}=\et{\liminf_k Y'_{q_k}}\leq \liminf_k \et{Y'_{q_k}}
\leq Y'_t,\end{equation}
 since $Y'$ is a supermartingale.
From (\ref{fin1}) and (\ref{fin2}) we get $Y'_t=Y_t$ a.s.
for $t\in [0,T]\setminus K$.
 \end{proof}

\begin{proof}[Proof ({\bf of  Corollary \ref{koro}})]\
If (\ref{chr}) holds, the Bayes rule for stochastic
processes and Fubini's theorem give
$\sup_{\QQ\in\MM}\EE_{\QQ}[C_T]\leq x$, where
$C_t=\int_0^t c(u)\,\mu(\du)$. Following \cite{6} we
define $F_t=\esssup_{\QQ\in\MM}\EE_{\QQ}[C_T|\FF_t]$. By
Theorem 2.1.1 in \cite{9}, a modification of $F_t$ can be
chosen in such a way to make $F_t$ a c\`{a}dl\`{a}g
supermartingale under each $\QQ\in\MM$. Furthermore,
$F_0=x'\leq x$. The Optional Decomposition Theorem
guarantees the existence of an $X\in\XX(x')$ and a
c\`{a}dl\`{a}g nonincreasing process $D$ with $D_0=0$
such that $F_t=X_t-D_t$. Since
$C_t=\EE_{\QQ}[C_t|\FF_t]\leq \EE_{\QQ}[C_T|\FF_t]$ we
conclude that $F_t\geq C_t$ and so $(x-x')+X_t-C_t\geq
X_t-D_t-C_t=F_t-C_t\geq 0$. Thus, $C_t$ is an
$x$-admissible consumption process.

Conversely, suppose $c$ is a consumption density of an absolutely
continuous consumption process $C_T$. It is easy to see that
$\sup_{\QQ\in\MM}\EE_{\QQ}[C_T]\leq x$ must hold, so
\[ \sup_{Y^{\QQ}\in\YY^{e}}\int_0^T
\EE[Y^{\QQ}_uc(u)]\,\mu(\du)=\sup_{Y^{\QQ}\in\YY^{e}}\EE\left[\int_0^T
Y^{\QQ}_uc(u)\,\mu(\du)\right]\leq x.\] Let $Y\in\YY$. By
Theorem \ref{myprop}, b), we can find a sequence
$(\ynn)_{n\in\N}$ in the solid hull of $\YY^e$ such that
$(\ynn)_{n\in\N}$ Fatou converges to $Y$. By the previous
lemma, $\liminf_n\ynn_t= Y_t$ for all $t$ in $[0,T]$,
except for maybe $t$ in a countable subset of $[0,T)$ and
thus for $t$ $\mu$-a.e. By Fatou lemma and monotonicity,
\begin{eqnarray*}\nonumber \EE[
\int_0^TY_uc(u)\,\mu(\du)]&=&\EE\left[\int_0^T
\liminf_n Y^{{(n)}}_uc(u)\,\mu(\du)\right]\leq\\
\nonumber &\leq& \liminf_n \EE\left[\int_0^T
Y^{{(n)}}_uc(u)\,\mu(\du)\right]\leq x.
\end{eqnarray*}
\end{proof}

{\bf Acknowledgments:} This work is a revised version of Chapter 1
of the author's Master's Thesis submitted to University of Zagreb,
Croatia, July 1999. I would like to thank Walter Schachermayer for
being a supportive advisor and conversations with whom gave birth
to this paper. Acknowledgements go to {\" O}sterreichische
Akademische Austauschdienst of the Austrian government and
Institut f\" ur Finanz- und Versicherungsmathematik, Technische
Universit\"at Wien, Austria for financial support. Many thanks to
Ioannis Karatzas, Zoran Vondra\v cek and the anonymous referee for
useful suggestions on presentation.

\providecommand{\bysame}{\leavevmode\hbox to3em{\hrulefill}\thinspace}


\begin{thebibliography}{KPR84}

\bibitem[BS99]{bipolar}
W.~Brannath and W.~{Schachermayer}, \emph{A bipolar theorem for subsets of
  {$L^0_+(\Omega,{\mathcal F},{\mathcal P})$}}, S\' eminaire de Probabilit\' es
  \textbf{XXXIII} (1999), 349--354.

\bibitem[DS93]{FT}
F.~Delbean and W.~Schachermayer, \emph{A general version of the fundamental
  theorem of asset pricing}, Mathematische Annalen \textbf{300} (1993),
  463--520.

\bibitem[DS98]{UNB}
F.~Delbaen and W.~Schachermayer, \emph{The fundamental theorem of asset pricing
  for unbounded stochastic processes}, Mathematische Annalen \textbf{312}
  (1998), 215--250.

\bibitem[DS99]{8}
F.~Delbean and W.~Schachermayer, \emph{A compactness principle for bounded
  sequences of martingales with applications}, Proceedings of the Seminar of
  Stochastic Analysis, Random Fields and Applications, 1999.

\bibitem[EQ95]{9}
N.~{El Karoui} and M.~C. Quenez, \emph{Dynamic programming and pricing of
  contingent claims in an incomplete market}, SIAM Journal on Control and
  Optimization \textbf{33/1} (1995), 29--66.

\bibitem[FK97]{7}
H.~F{\"o}llmer and D.~Kramkov, \emph{Optional decomposition under constraints},
  Probability Theory and Related Fields \textbf{109} (1997), 1--25.

\bibitem[FK98]{3}
H.~F{\"o}llmer and Yu.~M. Kabanov, \emph{Optional decomposition and lagrange
  multipliers}, Finance and Stochastics \textbf{2} (1998), 69--81.

\bibitem[KPR84]{kapr}
S.~Kalton, N.~T. Peck, and J.~W. Roberts, \emph{An {F}-space sampler}, London
  Math. Soc. Lecture Notes \textbf{89} (1984).

\bibitem[Kra96]{6}
D.~Kramkov, \emph{Optional decomposition of supermartingales and hedging
  contingent claims in incomplete security markets}, Probability Theory and
  Related Fields \textbf{105} (1996), 459--479.

\bibitem[KS91]{KSB}
Ioannis Karatzas and Steven~E. Shreve, \emph{Brownian {M}otion and {S}tochastic
  {C}alculus}, second ed., Springer-Verlag, New York, 1991.

\bibitem[KS98]{KS}
I.~Karatzas and Steven~E. Shreve, \emph{Methods of {M}athematical {F}inance},
  Springer, New York, 1998.

\bibitem[KS99]{4}
D.~Kramkov and W.~{Schachermayer}, \emph{A condition on the asymptotic
  elasticity of utility functions and optimal investment in incomplete
  markets}, Annals of Applied Probability \textbf{9} (1999), no.~3, 904--950.

\bibitem[RY91]{5}
D.~Revuz and M.~Yor, \emph{Continuous {M}artingales and {B}rownian {M}otion},
  Springer Verlag, New York, 1991.

\bibitem[Sch86]{S86}
M.~Schwartz, \emph{New proofs of a theorem of {K}oml{\' o}s}, Acta Math. Hung.
  \textbf{47} (1986), 181--185.

\bibitem[{\v Z}it99]{gz2}
Gordan {\v Z}itkovi{\' c}, \emph{Maximization of utility of consumption in
  incomplete semimartingale markets}, working paper, Tehnische Universit{\" a}t
  Wien (1999).

\end{thebibliography}
\end{document}